\documentclass{birkau}

\usepackage[utf8]{inputenc}
\usepackage[T1]{fontenc}
\usepackage{amsmath,amssymb,url}
\usepackage{enumitem}
\usepackage{needspace}
\numberwithin{equation}{section}

\makeatletter
\let\orig@enddoc@text\enddoc@text
\def\enddoc@text{\needspace{16\baselineskip}\orig@enddoc@text}
\makeatother

\theoremstyle{plain}
\newtheorem{theorem}{Theorem}[section]
\newtheorem{lemma}[theorem]{Lemma}
\newtheorem{proposition}[theorem]{Proposition}
\newtheorem{corollary}[theorem]{Corollary}
\newtheorem{example}[theorem]{Example}

\theoremstyle{definition}
\newtheorem{definition}[theorem]{Definition}

\newtheorem{remark}[theorem]{Remark}

\newtheorem{problem}[theorem]{Problem}

\newcommand{\N}{\mathbb{N}}
\newcommand{\R}{\mathbb{R}}
\newcommand{\dd}{^{d}}   % disjoint complement, used as X\dd

\title[Truncation bands, unitality, and homomorphisms]{Truncation bands, unitality, and\\ homomorphisms in truncated Riesz spaces:\\ Boulabiar's problems and beyond}

\author[M. Habibi]{Mohamed Habibi}
\address{Laboratoire de Recherche LATAO\\D\'epartement de Math\'ematiques\\Facult\'e des Sciences de Tunis\\Universit\'e de Tunis El Manar\\2092 El Manar\\Tunisia}
\email{mohamed.habibi@ipest.ucar.tn}

\corrauthor[H. Hafsi]{Hamza Hafsi}
\address{Preparatory Institute for Engineering Studies of Tunis\\University of Tunis\\1089 Tunis\\Tunisia}
\email{hafsi.hamza1@gmail.com,Hamza.Hafsi@ipeit.rnu.tn }
\subjclass{46A40, 06F20, 46B42, 47B65}

\keywords{truncated Riesz space, truncation band, projection property, canonical unitization, AM-space, truncation homomorphism}

\begin{document}

\begin{abstract}
We answer three questions from Boulabiar's survey on truncated Riesz spaces. First, the truncation band projection property coincides with the principal projection property, with no completeness or Archimedean hypothesis needed, and the band projection onto a truncation band is then given explicitly by the truncation itself. Second, a truncation on any Riesz space is unital exactly when its fixed-point set has a supremum. Hence bounded truncations on $KB$-spaces are unital, and a nonunital truncated Banach lattice with bounded truncation contains a closed sublattice copy of $c_0$ and is never an $AL$-space, while on the $AM$ side the smallest truncation unitization norm is always an $M$-norm, though the largest one need not be. Third, a truncation homomorphism whose domain truncation is Archimedean need not be a Riesz homomorphism, as an explicit counterexample built from $c_0$ and $c_0/c_{00}$ shows, but it is one modulo truncation infinitesimals whenever the codomain space is Archimedean.
\end{abstract}

\maketitle

% ==============================================================
% 1. INTRODUCTION
% ==============================================================
\section{Introduction}
\label{sec:intro}

A \emph{truncation} on a Riesz space $E$ with positive cone $E^{+}$ is a unary operation $\tau : E^{+} \to E^{+}$ satisfying the lattice-exchange identity $\tau(f) \wedge g = f \wedge \tau(g)$ for all $f, g \in E^{+}$. Ball \cite{Ball1,Ball2} introduced the notion as an intrinsic replacement for the classical device of an external distinguished weak unit. A truncated Riesz space $E \subseteq \R^X$ that is stable under $f \mapsto f \wedge \mathbf 1$, for a fixed function $\mathbf 1$, need not itself contain $\mathbf 1$. Yet the truncation alone is enough to recover a Yosida-type representation of $E$ by continuous functions that vanish appropriately at infinity, with no need to assume that $E$ has a unit.

This single axiom, refined and simplified in the form above, has since generated a substantial theory. A structure theorem realizes every truncation on an Archimedean Riesz space as $\tau(f) = pf + e \wedge f$ inside a universal completion \cite{Boulabiar-structure}. Representation theorems by continuous functions, together with Stone--Weierstrass type density results, are available in \cite{Boulabiar-survey}. Every truncated Riesz space $E$ whose truncation is weak, in the sense recalled in Section~\ref{subsec:truncations}, has a canonical (or Alexandroff) unitization $E \oplus \R$: a unital truncated Riesz space that contains $E$ as a maximal ideal and has a universal property among such unitizations \cite{Boulabiar-ElAdeb,Boulabiar-Hafsi-Mahfoudhi}. A normed refinement describes the extremal lattice norms on $E \oplus \R$ that extend a given norm on $E$ \cite{Boulabiar-Hafsi}. The transfer of order and completeness properties between $E$ and $E \oplus \R$ — the Archimedean property, relative uniform completeness, Dedekind and lateral completeness, universal completeness, and the projection property — is studied in \cite{Habibi-Hafsi}. Positive operators compatible with a truncation are studied in \cite{Boulabiar-Hajji-MJM}. Boulabiar's survey \cite{Boulabiar-survey} organizes this body of work and closes with six open problems.

The purpose of this paper is to answer three open problems from Boulabiar's survey \cite{Boulabiar-survey}, and to complement each answer with further results.

Our first contribution (Section~\ref{subsec:results-tbpp}) answers \cite[Problem~8.5]{Boulabiar-survey}. For a truncation $\tau$ on a Riesz space $E$, the band generated by its fixed-point set $[\tau]$ is called a \emph{truncation band}. Boulabiar asks what can be said of Riesz spaces in which every truncation band is a projection band. We show (Theorem~\ref{thm:tbpp-equals-ppp}) that this \emph{truncation band projection property} coincides exactly with the classical \emph{principal} projection property, with no completeness or Archimedean hypothesis needed. The band projection onto a truncation band $B_{[\tau]}$ is then given by the explicit formula $P_{B_{[\tau]}}(z) = \sup_n(z \wedge n\tau(z))$.

Our second contribution (Section~\ref{subsec:results-am-al}) addresses \cite[Problem~8.4]{Boulabiar-survey}. What can be said about truncated $AM$- and $AL$-spaces? Does either property transfer between $E$ and $E \oplus \R$? When $E$ is unital, the unitization splits as a direct sum of projection bands, and the question reduces to an elementary, essentially folklore fact about lattice norms on such direct sums (Lemma~\ref{lem:direct-sum-norms}). The real substance of the problem lies in the nonunital case, where no such decomposition exists.

We show that a truncation is unital exactly when its fixed-point set has a supremum (Theorem~\ref{thm:unitality-criterion}). It follows that every bounded truncation on a Banach lattice with the Levi property is unital. This covers, in particular, every $KB$-space and every dual Banach lattice. So a nonunital truncated Banach lattice with bounded truncation always contains a closed sublattice copy of $c_0$, and it can be neither reflexive, nor weakly sequentially complete, nor an $AL$-space (Theorem~\ref{thm:c0-embeds}). In particular, the canonical unitization of a truncated normed Riesz space is never an $AL$-space under a truncation unitization norm (Corollary~\ref{cor:never-AL}). Passing to the norm completion restores the $AL$ property, though the truncation unit then lies outside the original space (Theorem~\ref{thm:completion-unital}). On the $AM$ side, which is compatible with nonunitality, the smallest truncation unitization norm is always an $M$-norm (Theorem~\ref{thm:smallest-is-M-norm}). An explicit example on $C_0(\R)$ shows that the largest one need not be (Example~\ref{ex:S-vs-L}).

Our third contribution (Section~\ref{subsec:results-hom}) settles \cite[Problem~8.1]{Boulabiar-survey}. Boulabiar and Hajji proved that a truncation homomorphism $T : E \to F$ between truncated Riesz spaces is automatically a Riesz homomorphism whenever the truncation on the \emph{codomain} $F$ is Archimedean \cite{Boulabiar-Hajji-MJM}. They asked whether the hypothesis could instead be placed on the domain. We answer negatively, with an explicit, choice-free, surjective counterexample built from $c_0$ and the quotient $c_0/c_{00}$ (Theorem~\ref{thm:hom-counterexample}). We also show that the Boulabiar--Hajji theorem is optimal in a strong sense. A single explicit test domain detects the failure of the implication for every non-Archimedean codomain truncation, even if the conclusion is weakened to mere positivity or the domain hypothesis is strengthened (Theorem~\ref{thm:optimality}).

We then locate the precise dividing line inside the codomain. Every truncation $\tau$ on a Riesz space $F$ induces a truncation on the quotient of $F$ by the band $F_0(\tau)$ of truncation infinitesimal elements. Using the structure theorem of \cite{Boulabiar-structure} inside the universal completion of $F$, we prove that this induced truncation is always Archimedean once $F$ itself is Archimedean (Theorem~\ref{thm:main-hom}). Consequently, every truncation homomorphism into an Archimedean Riesz space is a Riesz homomorphism \emph{modulo truncation infinitesimals}, whatever the truncation on that codomain (Corollary~\ref{cor:modulo-infinitesimals}). We check a further case beyond the Archimedean one, and we record the open questions that remain (Problem~\ref{prob:conjecture-strong}).

The paper is organized as follows. Section~\ref{sec:prelim} collects the notation and background needed for all three contributions. Section~\ref{sec:results} states our results, organized into the three contributions described above. Section~\ref{sec:proofs} contains all proofs. Our notation and terminology for Riesz spaces and Banach lattices follow \cite{Aliprantis-Burkinshaw,Luxemburg-Zaanen,MeyerNieberg,Schaefer}. Unexplained notions and elementary properties of truncations can be found in \cite{Boulabiar-survey} and the references therein.

% ==============================================================
% 2. PRELIMINARIES AND NOTATION
% ==============================================================
\section{Preliminaries and notation}
\label{sec:prelim}

Throughout the paper, all Riesz spaces are real. We follow standard terminology and notation for Riesz spaces (vector lattices) and Banach lattices, as in \cite{Aliprantis-Burkinshaw,Luxemburg-Zaanen,MeyerNieberg,Schaefer}. In particular, for a nonempty subset $A$ of a Riesz space $E$ we write $A\dd$ for its disjoint complement, $B_A$ for the band generated by $A$, $A^{+} := A \cap E^{+}$, and, for $w \in E$, $B_w$ for the principal band generated by $w$. Below we collect only the notation, terminology, and facts about truncations that are used in more than one of the paper's three contributions. Any background needed for a single contribution is instead recalled at the point where we use it, in Section~\ref{sec:results}.

\subsection{Truncations and truncation bands}
\label{subsec:truncations}

Let $E$ be a Riesz space with positive cone $E^{+}$. A \emph{truncation} on $E$ is a unary operation $\tau : E^{+} \longrightarrow E^{+}$ satisfying
\[
\tau(f) \wedge g = f \wedge \tau(g) \qquad (f, g \in E^{+}). \tag{$\tau_1$}
\]
Two further axioms are considered when needed:
\[
\tau(f) = 0 \implies f = 0, \tag{$\tau_2$}
\]
\[
\tau(nf) = nf \text{ for all } n \in \N \implies f = 0. \tag{$\tau_3$}
\]
A truncation satisfying $(\tau_2)$ is called \emph{weak}, and one satisfying $(\tau_3)$ is called \emph{Archimedean}. A truncation $\tau$ is called \emph{strong} if for every $0 \ne f \in E^{+}$ there is $\lambda > 0$ with $\tau(\lambda f) = \lambda f$. A Riesz space equipped with a truncation is a \emph{truncated Riesz space}. The truncation is \emph{unital} if there is $u \in E^{+}$ with $\tau(f) = u \wedge f$ for all $f \in E^{+}$. Such $u$ is then unique, and we call it the \emph{truncation unit}.

We collect the elementary properties of truncations that we will use repeatedly. See \cite{Ball1,Ball2,Boulabiar-survey} and \cite[Lemma~2.2]{Habibi-Hafsi}.

\begin{lemma}
\label{lem:truncation-basics}
Let $\tau$ be a truncation on a Riesz space $E$ and let $f, g \in E^{+}$.
\begin{enumerate}[label=(\roman*)]
\item $\tau(f) = f \wedge \tau(\tau(f))$, so that $\tau(f) \le f$ and $\tau(\tau(f)) = \tau(f)$. Moreover $\tau(f \wedge g) = \tau(f) \wedge \tau(g)$.
\item $f \le g$ implies $\tau(f) \le \tau(g)$.
\item If $g \in [\tau]$, then $g \wedge f \le \tau(f)$.
\item $\tau(f) = 0$ if and only if $f$ is disjoint from $[\tau]$,
\end{enumerate}
where $[\tau] := \{f \in E^{+} : \tau(f) = f\}$ denotes the set of \emph{fixed points} of $\tau$.
\end{lemma}

The following consequences of Lemma~\ref{lem:truncation-basics}, the last of which is a Birkhoff type inequality, are used in Sections~\ref{subsec:results-am-al} and~\ref{subsec:results-hom}. We include the short proof for completeness.

\begin{lemma}
\label{lem:birkhoff}
Let $\tau$ be a truncation on a Riesz space $E$ and let $f, g \in E^{+}$.
\begin{enumerate}[label=(\roman*)]
\item If $f \le g$ and $g \in [\tau]$, then $f \in [\tau]$.
\item $\tau(f+g) \le \tau(f) + \tau(g)$.
\item $|\tau(f) - \tau(g)| \le \tau(|f-g|) \le |f-g|$.
\end{enumerate}
\end{lemma}

\begin{proof}
(i) By Lemma~\ref{lem:truncation-basics}(iii) applied to $g \in [\tau]$, $f = g \wedge f \le \tau(f) \le f$.

(ii) Put $a = \tau(f+g)$. By Lemma~\ref{lem:truncation-basics}(i), $a \in [\tau]$ and $0 \le a \le f+g$, so the Riesz decomposition property gives $a = a_1 + a_2$ with $0 \le a_1 \le f$ and $0 \le a_2 \le g$. By (i), $a_1, a_2 \in [\tau]$, so $a_1 = \tau(a_1) \le \tau(f)$ and $a_2 = \tau(a_2) \le \tau(g)$ by Lemma~\ref{lem:truncation-basics}(ii).

(iii) From $f \le g + |f-g|$ and (ii), $\tau(f) \le \tau(g) + \tau(|f-g|)$, and symmetrically $\tau(g) \le \tau(f) + \tau(|f-g|)$. Hence $|\tau(f)-\tau(g)| \le \tau(|f-g|)$, and $\tau(|f-g|) \le |f-g|$ by Lemma~\ref{lem:truncation-basics}(i).
\end{proof}

\begin{definition}
\label{def:truncation-band}
Let $\tau$ be a truncation on a Riesz space $E$. The band $B_{[\tau]}$ generated by $[\tau]$ is called the \emph{truncation band} of $\tau$. The Riesz space $E$ is said to have the \emph{truncation band projection property} if the truncation band of every truncation on $E$ is a projection band.
\end{definition}

\subsection{Projection properties}
\label{subsec:projection}

Recall that $E$ has the \emph{projection property} if every band of $E$ is a projection band, and the \emph{principal projection property} if every principal band of $E$ is a projection band. Every Dedekind complete Riesz space has the projection property \cite[Theorem~1.42]{Aliprantis-Burkinshaw}, and the projection property in turn implies the principal projection property. The principal projection property implies the Archimedean property \cite[Theorem~25.1]{Luxemburg-Zaanen}. If $E$ is Archimedean and $B$ is a principal projection band with band projection $P$, generated by $w \in E^{+}$, then
\[
Pz = \sup_{n} (z \wedge nw) \qquad (z \in E^{+}), \tag{1}
\]
see \cite[Theorem~1.47]{Aliprantis-Burkinshaw}. % TODO: vérifier le numéro de théorème dans l'édition 2006.
Moreover, an ideal $B$ of $E$ is a projection band if and only if, for every $z \in E^{+}$, the supremum $\sup\big(B^{+} \cap [0, z]\big)$ exists in $E$ and belongs to $B$ \cite[Theorem~1.41]{Aliprantis-Burkinshaw}.

\subsection{The canonical unitization}
\label{subsec:unitization}

Following \cite{Boulabiar-ElAdeb,Boulabiar-Hafsi-Mahfoudhi,Habibi-Hafsi}, every truncated Riesz space $E$ (with a truncation satisfying $(\tau_2)$) embeds as a maximal ideal, order dense when $E$ is nonunital, of a unital truncated Riesz space through its \emph{canonical} (or \emph{Alexandroff}) \emph{unitization} $E \oplus \R$, whose positive cone is
\[
(E \oplus \R)^{+} = E^{+} \cup \Big\{ x + \lambda : 0 < \lambda \in \R,\ x \in E,\ \tfrac{1}{\lambda} x^{-} \in [\tau] \Big\},
\]
the truncation on $E\oplus\R$ being the meet with the adjoined unit $\mathbf{1} := 0 + 1$. We will use two facts about $E \oplus \R$, both valid when the truncation satisfies $(\tau_2)$.

If $E$ is unital with truncation unit $u$, then $E$ and $\R(\mathbf 1 - u)$ are complementary projection bands in $E \oplus \R$. Every $x + r \in E \oplus \R$ then decomposes as
\[
x + r = (x+ru) + r(\mathbf 1 - u), \qquad \text{with} \qquad |x+r| = |x+ru| + |r|(\mathbf 1 - u).
\]
If instead $E$ is nonunital, then its disjoint complement in $E \oplus \R$ is trivial, and $E$ is order dense in $E \oplus \R$ (\cite[Theorem~2.1]{Habibi-Hafsi}, \cite{Boulabiar-Hafsi-Mahfoudhi}).

Finally, suppose $E$ is nonunital. Let $H$ be a Riesz space containing $E$ as a Riesz subspace, and let $w$ be a weak order unit of $H$ with $\tau(f) = w \wedge f$ for all $f \in E^{+}$, the infimum taken in $H$. Then $E + \R w$ is a Riesz subspace of $H$, and $x + r \mapsto x + rw$ is a Riesz isomorphism of $E \oplus \R$ onto $E + \R w$ that carries $\mathbf 1$ to $w$ (\cite[Theorem~3.1, Corollary~3.2]{Boulabiar-Hafsi-Mahfoudhi}). We will use this identification repeatedly below, without further comment, whenever a truncation is realized as the meet with a weak order unit in some ambient space.

\subsection{Truncated normed Riesz spaces and unitization norms}
\label{subsec:normed}

Let $(E,\tau)$ be a truncated Riesz space equipped with a lattice norm $\|\cdot\|_E$. The truncation $\tau$ is \emph{bounded} if $\sup\{\|\tau(f)\|_E : f \in E^{+}\}$ is finite. A \emph{truncated normed Riesz space} is a triple $(E, \|\cdot\|_E, \tau)$ consisting of a Riesz space, a lattice norm, and a bounded truncation satisfying $(\tau_2)$, normalized so that
\[
\sup\{\|\tau(f)\|_E : f \in E^{+}\} = 1.
\]
This normalization is harmless, since it only rescales the norm, and it reads $\|u\|_E = 1$ when $\tau$ is unital with truncation unit $u$. The axiom $(\tau_2)$ is the standing hypothesis of \cite{Boulabiar-Hafsi-Mahfoudhi}, from which the facts about $E \oplus \R$ recalled in Section~\ref{subsec:unitization} are taken, and it is not implied by boundedness: on $c_0 \oplus_\infty c_0$, the truncation $\tau(f,g) = (\mathbf 1 \wedge f, 0)$ is bounded and normalized but fails $(\tau_2)$. A \emph{truncated Banach lattice} is a Banach lattice equipped with a truncation, no boundedness or normalization being assumed unless stated. A \emph{truncation unitization norm} on $E \oplus \R$ is a lattice norm $\|\cdot\|_U$ such that $\|\mathbf 1\|_U = 1$ and $\|x\|_U = \|x\|_E$ for every $x \in E$.

\begin{theorem}[{\cite{Boulabiar-Hafsi}, see also \cite[Theorem~6.3, 6.4]{Boulabiar-survey}}]
\label{thm:extreme-norms}
Let $(E, \|\cdot\|_E, \tau)$ be a truncated normed Riesz space. Then
\[
\|x+r\|_L = \big\| \big( |x+r| - |r|\mathbf 1 \big)^{+} \big\|_E + |r|
\]
defines the largest truncation unitization norm on $E \oplus \R$. If moreover $E$ is nonunital, then
\[
\|x+r\|_S = \sup\{ \|g\|_E : g \in E \text{ and } |g| \le |x+r| \}
\]
defines the smallest truncation unitization norm on $E \oplus \R$. Furthermore, $E$ is a Banach lattice if and only if $E \oplus \R$ is a Banach lattice under one, equivalently under every, truncation unitization norm.
\end{theorem}

Recall that a lattice norm is \emph{additive on disjoint positive elements} if $\|x+y\| = \|x\| + \|y\|$ whenever $x \wedge y = 0$ (such $x,y$ being automatically positive), and is an \emph{M-norm} if $\|x \vee y\| = \|x\| \vee \|y\|$ whenever $x \wedge y = 0$. An \emph{AL-space} (resp.\ an \emph{AM-space}) is a Banach lattice whose norm is additive on disjoint positive elements (resp.\ is an M-norm). For a Banach lattice, these conditions on disjoint elements are equivalent to the classical ones, namely $\|x+y\| = \|x\|+\|y\|$ and $\|x \vee y\| = \|x\| \vee \|y\|$ for all $x, y \in E^{+}$, by Kakutani's representation theorems, see \cite[Theorems~1.b.2 and~1.b.6]{Lindenstrauss-Tzafriri}, so that the results on $AL$- and $AM$-spaces quoted from \cite{MeyerNieberg} apply. % TODO: vérifier les numéros 1.b.2 (abstract L_p spaces) et 1.b.6 (abstract M spaces) dans Lindenstrauss--Tzafriri II.
A \emph{KB-space} is a Banach lattice in which every norm bounded increasing sequence of positive elements is norm convergent. A Banach lattice has the \emph{Levi property} if every upward directed, norm bounded subset of its positive cone has a supremum \cite{AliprantisBurkinshawLS}. % TODO: indiquer le numéro de la définition dans [2].

We record the classical facts used below.

\begin{theorem}
\label{thm:classical}
Let $E$ be a Banach lattice.
\begin{enumerate}[label=(\roman*)]
\item If $E$ is a $KB$-space, every norm bounded, upward directed subset of $E^{+}$ has a supremum in $E$, to which it converges in norm. In particular, every $KB$-space has the Levi property \cite{AliprantisBurkinshawLS}. % TODO: indiquer le numéro de théorème dans Aliprantis--Burkinshaw (Locally Solid), ou citer Meyer-Nieberg Theorem 2.4.12 (KB = bande dans le bidual) avec l'ordre-continuité de la norme.
\item $E$ is a $KB$-space if and only if $E$ contains no closed Riesz subspace lattice isomorphic to $c_0$, if and only if $E$ is weakly sequentially complete \cite[Theorem~2.5.6]{MeyerNieberg}. % TODO: vérifier les numéros cités de Meyer-Nieberg (Thm 2.5.6, Thm 2.4.15, Cor. 2.4.13, Prop. 2.4.19(ii)).
\item If $E$ is reflexive, then $E$ is a $KB$-space \cite[Theorem~2.4.15]{MeyerNieberg}. If $E$ is an $AL$-space, then $E$ is also a $KB$-space \cite[Corollary~2.4.13]{MeyerNieberg}.
\item If $E$ is a dual Banach lattice, then $E$ has the Levi property \cite[Proposition~2.4.19(ii)]{MeyerNieberg}.
\end{enumerate}
\end{theorem}

\subsection{Truncation homomorphisms and truncation infinitesimals}
\label{subsec:homomorphisms}

An \emph{extended truncation} on a Riesz space $E$ is a map $\sigma : E \to E$ such that $\sigma(0)=0$ and $\sigma(f)\wedge g = f \wedge \sigma(g)$ for all $f,g \in E$. By \cite[Theorem~2.4]{Boulabiar-survey}, restriction $\sigma \mapsto \sigma|_{E^+}$ is a bijection from extended truncations onto truncations on $E$, with inverse taking $\tau$ to
\[
\widetilde\tau(f) = \tau(f^{+}) - f^{-} \qquad (f \in E).
\]
Given truncated Riesz spaces $(E,\tau_E)$ and $(F,\tau_F)$, a linear operator $T : E \to F$ is a \emph{truncation homomorphism} if $T \circ \widetilde{\tau_E} = \widetilde{\tau_F} \circ T$.

For a truncation $\tau$ on a Riesz space $F$, the set
\[
F_0(\tau) = \{ g \in F : \tau(n|g|) = n|g| \text{ for all } n \in \N \}
\]
is called the band of \emph{truncation infinitesimal} elements of $F$. It is indeed a band, and $\tau$ is Archimedean if and only if $F_0(\tau) = \{0\}$ \cite[Proposition~4.1]{Boulabiar-survey}. The identity map of $F$ satisfies $(\tau_1)$ trivially, and it lies at the opposite end of the Archimedean spectrum.

\begin{lemma}
\label{lem:identity-truncation}
On any Riesz space $F$, the identity truncation $\tau_F(g) = g$ satisfies $\widetilde{\tau_F} = \mathrm{id}_F$ and $F_0(\tau_F) = F$. In particular, $\tau_F$ is weak, and it is Archimedean only if $F = \{0\}$.
\end{lemma}

\begin{proof}
For $g \in F$, $\widetilde{\tau_F}(g) = \tau_F(g^{+}) - g^{-} = g$, and $\tau_F(n|g|) = n|g|$ for all $n \in \N$, so $g \in F_0(\tau_F)$. Finally, $\tau_F(g) = 0$ forces $g = 0$.
\end{proof}

\begin{theorem}[Boulabiar--Hajji, {\cite[Theorem~2.2]{Boulabiar-Hajji-MJM}}]
\label{thm:BH}
Let $E, F$ be Riesz spaces with truncations. If the truncation on $F$ is Archimedean, then every truncation homomorphism $T : E \to F$ is a Riesz homomorphism (in particular, a positive operator).
\end{theorem}

This is the theorem behind \cite[Problem~8.1]{Boulabiar-survey}, which asks whether the Archimedean hypothesis can instead be imposed on the \emph{domain}: is every truncation homomorphism from a Riesz space with an Archimedean truncation into a Riesz space with an arbitrary truncation necessarily a Riesz homomorphism?

% ==============================================================
% 3. MAIN RESULTS
% ==============================================================
\section{Main results}
\label{sec:results}

This section presents our results, organized into three parts that match the paper's three contributions, in the order announced in the introduction. Section~\ref{subsec:results-tbpp} answers \cite[Problem~8.5]{Boulabiar-survey}. Section~\ref{subsec:results-am-al} answers \cite[Problem~8.4]{Boulabiar-survey}. Section~\ref{subsec:results-hom} answers \cite[Problem~8.1]{Boulabiar-survey}. The three problems are logically independent, and the three parts can be read in any order.

\subsection{Truncation bands and the principal projection property}
\label{subsec:results-tbpp}

We first record how truncation bands sit between principal bands and projection bands.

\needspace{7\baselineskip}
\begin{proposition}
\label{prop:bands}
Let $E$ be a Riesz space.
\begin{enumerate}[label=(\roman*)]
\item Every principal band of $E$ is a truncation band: for $w \in E^{+}$ one has $B_w = B_{[\tau_w]}$, where $\tau_w(f) = w \wedge f$ for all $f \in E^{+}$.
\item Every projection band of $E$ is a truncation band: if $B$ is a projection band with band projection $P$, then the restriction of $P$ to $E^{+}$ is a truncation on $E$ with $[P] = B^{+}$, so that $B_{[P]} = B$.
\item The projection property implies the truncation band projection property, which in turn implies the principal projection property.
\end{enumerate}
\end{proposition}

Our main result shows that the last implication in Proposition~\ref{prop:bands}(iii) can be reversed, thereby answering \cite[Problem~8.5]{Boulabiar-survey}.

\begin{theorem}
\label{thm:tbpp-equals-ppp}
A Riesz space $E$ has the truncation band projection property if and only if it has the principal projection property. Moreover, if $E$ has the principal projection property, then for every truncation $\tau$ on $E$ and every $z \in E^{+}$, the band projection onto $B_{[\tau]}$ is computed by the truncation itself:
\[
P_{B_{[\tau]}}(z) = \sup_{n} \big( z \wedge n\,\tau(z) \big).
\]
\end{theorem}

In particular, in a Riesz space with the principal projection property, the truncation bands are exactly the projection bands, and no Archimedean or completeness assumption is needed for this equivalence, the Archimedean property being in any case a consequence of the principal projection property (Section~\ref{subsec:projection}).

\subsection{Truncated $AM$- and $AL$-spaces (Problem~8.4)}
\label{subsec:results-am-al}

Boulabiar's \cite[Problem~8.4]{Boulabiar-survey} asks what can be said about truncated $AM$- and $AL$-spaces, and in particular whether the $AM$ and $AL$ properties transfer between a truncated normed Riesz space $E$ and its canonical unitization $E \oplus \R$. We show that the two halves of this question behave quite differently. Its substance lies entirely on the nonunital side. The unital case reduces to an elementary, presumably well-known fact about direct sums of projection bands.

\subsubsection{The unital case}
\label{subsubsec:unital-am-al}

Suppose $E$ is unital with truncation unit $u$, so that $E \oplus \R = E \oplus \R(\mathbf 1 - u)$ is the direct sum of the projection bands $E$ and $\R(\mathbf 1-u)$ (Section~\ref{subsec:unitization}). The transfer of the $AL$, $AM$, and $KB$ properties is then an instance of the following elementary lemma about lattice norms on a direct sum of two projection bands, which we record for completeness.

\begin{lemma}
\label{lem:direct-sum-norms}
Let $X = Y \oplus Z$ be the direct sum of two projection bands $Y, Z$ of a Riesz space $X$.
\begin{enumerate}[label=(\roman*)]
\item The lattice norms on $X$ making it an $AL$-space are exactly the norms $\|y+z\| = \|y\|_Y + \|z\|_Z$, where $\|\cdot\|_Y$ and $\|\cdot\|_Z$ are $AL$-norms on $Y$ and $Z$.
\item The lattice norms on $X$ making it an $AM$-space are exactly the norms $\|y+z\| = \|y\|_Y \vee \|z\|_Z$, where $\|\cdot\|_Y$ and $\|\cdot\|_Z$ are $AM$-norms on $Y$ and $Z$.
\item A lattice norm on $X$ makes it a $KB$-space if and only if its restrictions to $Y$ and $Z$ make $Y$ and $Z$ into $KB$-spaces.
\end{enumerate}
\end{lemma}

Applying this with $Y = E$ and $Z = \R(\mathbf 1-u)$, on which every norm is automatically both an $AL$- and an $AM$-norm, Lemma~\ref{lem:direct-sum-norms} shows three things.
\begin{itemize}
\item $E$ is an $AL$-space if and only if $E \oplus \R$ is an $AL$-space under the norm
\[
\|x+r\| = \|x+ru\|_E + |r|,
\]
for some (every) choice of $|r|$-coefficient.
\item $E$ is an $AM$-space if and only if $E \oplus \R$ is an $AM$-space under
\[
\|x+r\| = \|x+ru\|_E \vee \lambda|r|
\]
for some $\lambda>0$, and this is a truncation unitization norm exactly when $0 < \lambda \le 1$.
\item And $E$ is a $KB$-space if and only if $E \oplus \R$ is a $KB$-space under any lattice norm extending $\|\cdot\|_E$.
\end{itemize}
We omit the straightforward verification. Note that no $AL$-norm extension of $\|\cdot\|_E$ is ever a truncation unitization norm, since it would force
\[
\|\mathbf 1\|_U = \|u\|_E + \|\mathbf 1 - u\|_{\R(\mathbf 1-u)} > \|u\|_E = 1.
\]

\subsubsection{The nonunital case}
\label{subsubsec:nonunital-am-al}

When $E$ is nonunital, $E$ has trivial disjoint complement in $E \oplus \R$ and is order dense in it, so no band decomposition is available and Lemma~\ref{lem:direct-sum-norms} does not apply. This is where the genuine content of Problem~8.4 lies.

Our first observation is a purely order-theoretic criterion for unitality.

\begin{theorem}
\label{thm:unitality-criterion}
Let $\tau$ be a truncation on a Riesz space $E$. Then $\tau$ is unital if and only if $[\tau]$ has a supremum in $E$, in which case the truncation unit is $\sup[\tau]$. Consequently, if $E$ is a Banach lattice with the Levi property and $\tau$ is bounded, then $\tau$ is unital. In particular, every bounded truncation on a $KB$-space, or on a dual Banach lattice such as $\ell^\infty$ or $L^\infty(\mu)$ for $\sigma$-finite $\mu$, is unital.
\end{theorem}

Theorem~\ref{thm:unitality-criterion} is a normed counterpart of the fact that every Archimedean truncation on a universally complete Riesz space is unital \cite[Theorem~4.5]{Boulabiar-survey}, universal completeness and the Archimedean axiom on $\tau$ being replaced, respectively, by the Levi property and by boundedness of $\tau$. It identifies $c_0$ as essentially the only obstruction to unitality among truncated Banach lattices.

\begin{theorem}
\label{thm:c0-embeds}
Let $E$ be a nonunital truncated Banach lattice whose truncation is bounded. Then $E$ contains a closed Riesz subspace lattice isomorphic to $c_0$. In particular, $E$ is neither reflexive nor weakly sequentially complete, and $E$ is not an $AL$-space.
\end{theorem}

The $AL$ half of \cite[Problem~8.4]{Boulabiar-survey} can now be settled completely: the $AL$ property never passes to the canonical unitization under a truncation unitization norm.

\begin{corollary}
\label{cor:never-AL}
Let $(E, \|\cdot\|_E, \tau)$ be a truncated normed Riesz space. Then $E \oplus \R$ is not an $AL$-space under any truncation unitization norm.
\end{corollary}

Norm completion changes the picture on the $AL$ side. A bounded truncation on an $AL$-space is always unital by Theorem~\ref{thm:c0-embeds}. Yet nonunital truncated normed Riesz spaces whose norm is additive on disjoint positive elements do exist (Example~\ref{ex:completion} below), and their completion absorbs the truncation unit.

\begin{theorem}
\label{thm:completion-unital}
Let $(E, \|\cdot\|_E, \tau)$ be a truncated normed Riesz space with norm completion $\overline E$. Then $\tau$ extends uniquely to a truncation $\overline\tau$ on $\overline E$ with
\[
\sup\{\|\overline\tau(f)\| : f \in \overline E^{+}\} = 1.
\]
If moreover $E$ is nonunital and $\|\cdot\|_E$ is additive on disjoint positive elements, then $(\overline E, \overline\tau)$ is a unital truncated $AL$-space whose truncation unit lies in $\overline E \setminus E$.
\end{theorem}

\begin{example}
\label{ex:completion}
Let $E = \{f \in C([0,1]) : f(0) = 0\}$ with $\|f\|_E = \int_0^1 |f(t)|\,dt$ and $\tau(f) = \mathbf 1_{[0,1]} \wedge f$. Then $(E, \|\cdot\|_E, \tau)$ is a nonunital truncated normed Riesz space whose norm is additive on disjoint positive elements, and $E$ is not norm complete. Its completion is $\overline E = L^1([0,1])$, the extended truncation is again $\overline\tau(f) = \mathbf 1_{[0,1]} \wedge f$, and the truncation unit $\mathbf 1_{[0,1]}$ lies in $L^1([0,1]) \setminus E$, illustrating Theorem~\ref{thm:completion-unital}.
\end{example}

By contrast, the $AM$ property is perfectly compatible with nonunitality, $c_0$ being the standard example. We show that the smallest truncation unitization norm is always an $M$-norm in this case, while the largest one need not be.

\begin{theorem}
\label{thm:smallest-is-M-norm}
Let $(E, \|\cdot\|_E, \tau)$ be a nonunital truncated normed Riesz space whose norm is an $M$-norm. Then the smallest truncation unitization norm $\|\cdot\|_S$ (Theorem~\ref{thm:extreme-norms}) is an $M$-norm on $E \oplus \R$.
\end{theorem}

\begin{example}
\label{ex:S-vs-L}
\begin{enumerate}[label=(\roman*)]
\item Let $E = c_0$ with the supremum norm and $\tau(f) = \mathbf 1 \wedge f$ (pointwise meet with the constant sequence $\mathbf 1$). Identifying $E \oplus \R$ with $c$, the space of convergent sequences, via $x + r \mapsto x + r\mathbf 1$, one checks that $\|\cdot\|_S$ and $\|\cdot\|_L$ both coincide with the supremum norm of $c$. Hence \emph{all} truncation unitization norms on $c_0 \oplus \R$ coincide, and Theorem~\ref{thm:smallest-is-M-norm} reduces to the classical fact that the supremum norm of $c$ is an $M$-norm.
\item Let $E = C_0(\R)$ with the supremum norm, and let $e \in C_b(\R)$ be defined by $e(t) = 1$ for $t \le 0$ and $e(t) = (1+t)^{-1}$ for $t \ge 0$, so $0 < e \le 1$, $\sup e = 1$, and $e \notin C_0(\R)$. Then $\tau(f) = e \wedge f$ is a bounded, nonunital truncation on $E$ satisfying $(\tau_2)$, and $e$ identifies $E \oplus \R$ with $E + \R e \subset C_b(\R)$. There exists $\varphi \in (E\oplus\R)^{+}$ with $\|\varphi\|_S \ne \|\varphi\|_L$, and there exist disjoint $x, \psi \in (E\oplus\R)^{+}$ with $\|x \vee \psi\|_L > \|x\|_L \vee \|\psi\|_L$. So $\|\cdot\|_S$ and $\|\cdot\|_L$ are distinct, though equivalent, truncation unitization norms on $E \oplus \R$, and $\|\cdot\|_L$ fails to be an $M$-norm. This shows that Theorem~\ref{thm:smallest-is-M-norm} does not extend to arbitrary truncation unitization norms.
\end{enumerate}
\end{example}

\begin{problem}
\label{prob:unique-norm}
Characterize the nonunital truncated normed Riesz spaces whose norm is an $M$-norm and for which $\|\cdot\|_S = \|\cdot\|_L$, i.e., for which the truncation unitization norm is unique. Example~\ref{ex:S-vs-L}(i) shows this holds for $c_0$, while part (ii) shows it may fail. We regard this question as the part of \cite[Problem~8.4]{Boulabiar-survey} that remains open.
\end{problem}

\subsection{Truncation homomorphisms with Archimedean domain (Problem~8.1)}
\label{subsec:results-hom}

We answer \cite[Problem~8.1]{Boulabiar-survey} negatively by an explicit, choice-free, surjective counterexample, show that Theorem~\ref{thm:BH} is optimal in a strong sense, and then locate the exact dividing line inside the codomain: every truncation homomorphism into an \emph{Archimedean} Riesz space is a Riesz homomorphism \emph{modulo truncation infinitesimals}, whatever the truncation on that codomain.

\subsubsection{A negative answer}
\label{subsubsec:hom-negative}

\begin{lemma}
\label{lem:c0-hom}
Equip $E = c_0$ with $\tau_E(f) = f \wedge \mathbf 1$ (meet computed in $\R^\N$). Then $\tau_E$ is a strong Archimedean truncation on $c_0$, and $f - \widetilde{\tau_E}(f) = (f^{+}-\mathbf 1)^{+} \in c_{00}$ for every $f \in c_0$.
\end{lemma}

\begin{theorem}
\label{thm:hom-counterexample}
Let $E = c_0$ carry the truncation $\tau_E(f) = f \wedge \mathbf 1$, and let $F = c_0/c_{00}$ carry the identity truncation $\tau_F(g)=g$. Then $T = -\pi : c_0 \to c_0/c_{00}$ (where $\pi$ is the quotient map) is a surjective truncation homomorphism which is not a Riesz homomorphism, indeed not even a positive operator. In particular, \cite[Problem~8.1]{Boulabiar-survey} has a negative answer.
\end{theorem}

\begin{remark}
\label{rem:hom-mechanism}
When $F$ carries the identity truncation, $\widetilde{\tau_F} = \mathrm{id}_F$, and the intertwining condition collapses to $T$ vanishing on
\[
N_E := \operatorname{span}\{g - \tau_E(g) : g \in E^{+}\}.
\]
For $E=c_0$ as above, $N_E = c_{00}$. The Archimedean hypothesis on $\tau_E$ plays no protective role here. It never enters the intertwining condition, which is governed entirely by $\widetilde{\tau_F}$. The truncation $\tau_E$ above is even strong (Lemma~\ref{lem:c0-hom}). The identity truncation on the codomain, on the other hand, is weak but not Archimedean (Lemma~\ref{lem:identity-truncation}). So the negative answer persists even if ``Archimedean'' is strengthened to ``strong Archimedean'' on the domain, and ``arbitrary truncation'' is weakened to ``weak truncation'' on the codomain. The whole construction is choice-free.
\end{remark}

\subsubsection{Optimality of Theorem~\ref{thm:BH}}
\label{subsubsec:hom-optimal}

Let $u = (1,\tfrac12,\tfrac13,\dots) \in c_0^{+}$ and $D = \R u \oplus c_{00} \subseteq c_0$. Since $u$ has full support, every element of $D$ is uniquely $\alpha u + g$ with $\alpha \in \R$ and $g \in c_{00}$. So $\lambda(\alpha u + g) := \alpha$ defines a linear functional on $D$ that vanishes on $c_{00}$, with $\lambda(u)=1$.

\begin{lemma}
\label{lem:test-domain}
$D$ is a Riesz subspace of $c_0$, and $\tau_D(f) = f \wedge \mathbf 1$ is a strong Archimedean truncation on $D$ with $f - \widetilde{\tau_D}(f) = (f^{+}-\mathbf1)^{+} \in c_{00}$ for all $f \in D$.
\end{lemma}

\begin{theorem}
\label{thm:rank-one}
Let $F$ be a Riesz space with a non-Archimedean truncation $\tau_F$, and pick $0 \ne g \in F_0(\tau_F)$, $g_0 = |g| > 0$. Then $T : D \to F$, $T(f) = -\lambda(f) g_0$, is a truncation homomorphism from $(D,\tau_D)$ into $(F,\tau_F)$ that is neither a Riesz homomorphism nor a positive operator.
\end{theorem}

\begin{theorem}
\label{thm:optimality}
For a truncated Riesz space $(F,\tau_F)$, the following are equivalent.
\begin{enumerate}[label=(\roman*)]
\item $\tau_F$ is Archimedean.
\item Every truncation homomorphism $T : E \to F$, for every truncated Riesz space $(E,\tau_E)$, is a Riesz homomorphism.
\item The same statement holds with ``positive operator'' in place of ``Riesz homomorphism.''
\item Every truncation homomorphism from $(D,\tau_D)$ into $F$ is a Riesz homomorphism.
\item The same statement holds with ``positive operator.''
\end{enumerate}
\end{theorem}

Thus a single explicit test domain $(D,\tau_D)$ — Archimedean, carrying a strong Archimedean truncation — detects every failure of Archimedeanness in the codomain: no strengthening of the hypothesis on the domain in \cite[Problem~8.1]{Boulabiar-survey}, nor any weakening of the conclusion to mere positivity, can restore the implication. (Quantifying over the domain is essential: for a \emph{fixed} $E = \{0\}$, every operator into $F$ is vacuously a Riesz homomorphism regardless of $\tau_F$.)

\subsubsection{The dividing line: Riesz homomorphisms modulo truncation infinitesimals}
\label{subsubsec:hom-quotient}

Throughout, $(F,\tau)$ is a truncated Riesz space, $B = F_0(\tau)$, and $q : F \to F/B$ the quotient Riesz homomorphism.

\begin{proposition}
\label{prop:induced-truncation}
There is a unique truncation $\overline\tau$ on $F/B$ with $\overline\tau(q(g)) = q(\tau(g))$ for $g \in F^{+}$, and $q$ is then simultaneously a Riesz homomorphism and a truncation homomorphism from $(F,\tau)$ onto $(F/B, \overline\tau)$.
\end{proposition}

\begin{lemma}
\label{lem:reduction}
If $\overline\tau$ on $F/B$ is Archimedean, then for every truncated Riesz space $(E,\tau_E)$ and every truncation homomorphism $T : E \to F$, the composition $q \circ T$ is a Riesz homomorphism.
\end{lemma}

Our main theorem shows that the hypothesis of Lemma~\ref{lem:reduction} always holds when $F$ itself is Archimedean, using the structure theorem of \cite[Theorem~3.8]{Boulabiar-structure}: on an Archimedean Riesz space $F$ with truncation $\tau$, embedded order densely in its universal completion $F^u$ (an $f$-algebra with multiplicative identity a fixed weak order unit $w$ of $F^u$), there are a component $p$ of $w$ and $e \in (F^u)^{+}$ disjoint from $p$ with $\tau(f) = pf + e \wedge f$ for $f \in F^{+}$.

\begin{theorem}
\label{thm:main-hom}
If $F$ is Archimedean, then for every truncation $\tau$ on $F$, the induced truncation $\overline\tau$ on $F/F_0(\tau)$ is Archimedean.
\end{theorem}

\begin{corollary}
\label{cor:modulo-infinitesimals}
Let $F$ be Archimedean with an arbitrary truncation $\tau_F$, and $q : F \to F/F_0(\tau_F)$ the quotient map. For every truncated Riesz space $(E,\tau_E)$ and every truncation homomorphism $T : E \to F$, the composition $q \circ T$ is a Riesz homomorphism. Equivalently, $T(f\vee g) - T(f)\vee T(g) \in F_0(\tau_F)$ for all $f,g \in E$, and, in particular, $T(f)^{-} \in F_0(\tau_F)$ for all $f \in E^{+}$. In short, every truncation homomorphism into an Archimedean Riesz space is a Riesz homomorphism, and in particular positive, modulo truncation infinitesimals.
\end{corollary}

When $F$ and $\tau_F$ are both Archimedean,
\[
F_0(\tau_F)=\{0\},
\]
so $q$ is the identity and Corollary~\ref{cor:modulo-infinitesimals} reduces to Theorem~\ref{thm:BH}. It is also consistent with the fact that the pathological examples above (and \cite[Example~2.9]{Boulabiar-survey}) all have range inside $F_0(\tau_F)$.

\begin{problem}
\label{prob:conjecture-strong}
Does Theorem~\ref{thm:main-hom} hold without the Archimedean hypothesis on $F$, i.e., is the induced truncation $\overline\tau$ on $F/F_0(\tau)$ Archimedean for \emph{every} truncated Riesz space $(F,\tau)$?
\end{problem}

\begin{proposition}
\label{prop:totally-ordered}
Problem~\ref{prob:conjecture-strong} has a positive answer whenever $\tau = \tau_v$ is unital and $F$ is totally ordered (a case not covered by Theorem~\ref{thm:main-hom}, since a totally ordered Riesz space of dimension $\ge 2$ is never Archimedean).
\end{proposition}

Finally, we record a condition on the operator itself (rather than on the codomain alone) guaranteeing a genuine Riesz homomorphism.

\begin{proposition}
\label{prop:ideal-condition}
Let $T : (E,\tau_E) \to (F,\tau_F)$ be a truncation homomorphism and let $I$ be the ideal of $F$ generated by $T(E)$. If $I \cap F_0(\tau_F) = \{0\}$, then $T$ is a Riesz homomorphism.
\end{proposition}

% ==============================================================
% 4. PROOFS
% ==============================================================
\section{Proofs}
\label{sec:proofs}

\subsection{Proofs for Section~\ref{subsec:results-tbpp}}

\begin{proof}[Proof of Proposition~\ref{prop:bands}]
(i) The map $\tau_w$ is a truncation, since $(w \wedge f) \wedge g = f \wedge (w \wedge g)$ for all $f,g \in E^+$, and its set of fixed points is $[\tau_w] = \{f \in E^{+} : f \le w\}$. The band generated by this set is precisely the principal band $B_w$.

(ii) We first record the elementary identity: if $a, b, c \in E^{+}$ and $a \wedge c = 0$, then
\[
a \wedge (b + c) = a \wedge b. \tag{2}
\]
Let $f, g \in E^{+}$ and write $f = f_1 + f_2$, $g = g_1 + g_2$ with $f_1, g_1 \in B^{+}$ and $f_2, g_2 \in (B\dd)^{+}$. Since $f_1 \wedge g_2 = 0 = g_1 \wedge f_2$, two applications of (2) give
\[
P(f) \wedge g = f_1 \wedge (g_1 + g_2) = f_1 \wedge g_1 = (f_1+f_2)\wedge g_1 = f \wedge P(g),
\]
so $P$ restricts to a truncation on $E$. Clearly $P(f) = f$ if and only if $f \in B$, so $[P] = B^{+}$ and $B_{[P]} = B$.

(iii) If $E$ has the projection property, every band of $E$, in particular every truncation band, is a projection band. If $E$ has the truncation band projection property, then by (i) every principal band of $E$ is a projection band.
\end{proof}

\begin{proof}[Proof of Theorem~\ref{thm:tbpp-equals-ppp}]
Necessity is contained in Proposition~\ref{prop:bands}(iii). For sufficiency, assume that $E$ has the principal projection property, let $\tau$ be a truncation on $E$, and put $B = B_{[\tau]}$. By the criterion for projection bands recalled in Section~\ref{subsec:projection}, it suffices to show that, for every $z \in E^{+}$, the supremum $\sup\big(B^{+} \cap [0,z]\big)$ exists in $E$ and belongs to $B$. Fix $z \in E^{+}$.

\emph{Step 1: the band generated by $\tau(z)$.} The principal band $B_{\tau(z)}$ is a projection band, by the principal projection property. Write $P_z$ for its band projection. As a projection band, $B_{\tau(z)}$ equals its second disjoint complement, so taking disjoint complements in $\{\tau(z)\} \subseteq B_{\tau(z)} \subseteq \{\tau(z)\}^{dd}$ gives
\[
B_{\tau(z)}\dd = \{\tau(z)\}\dd, \qquad B_{\tau(z)} = \{\tau(z)\}^{dd}.
\]

\emph{Step 2: the trace on $[0,z]$.} We claim that
\[
B^{+} \cap [0,z] = B_{\tau(z)}^{+} \cap [0,z]. \tag{3}
\]
The inclusion $\supseteq$ is clear since $\tau(z) \in [\tau]$ (Lemma~\ref{lem:truncation-basics}(i)), so $B_{\tau(z)} \subseteq B$.

For the converse, let $b \in B^{+}$ with $b \le z$, let $h \in E$ be disjoint from $\tau(z)$, and put $c = |h| \wedge b$. From $c \le b \le z$ and Lemma~\ref{lem:truncation-basics}(ii) we get $\tau(c) \le \tau(z)$, while $\tau(c) \le c \le |h|$ by Lemma~\ref{lem:truncation-basics}(i). Hence
\[
\tau(c) \le |h| \wedge \tau(z) = 0,
\]
and Lemma~\ref{lem:truncation-basics}(iv) shows that $c$ is disjoint from $[\tau]$. Since $[\tau] \subseteq B \subseteq [\tau]^{dd}$, taking disjoint complements gives $B\dd = [\tau]\dd$, so $c \in B\dd$. As also $0 \le c \le b \in B$ and $B$ is an ideal, $c \in B$. Therefore $c \in B \cap B\dd = \{0\}$, i.e., $|h| \wedge b = 0$. Since $h$ was an arbitrary element disjoint from $\tau(z)$, it follows that $b \in \{\tau(z)\}^{dd} = B_{\tau(z)}$, proving (3).

\emph{Step 3: conclusion.} The element $P_z(z)$ lies in $B_{\tau(z)}^{+} \cap [0,z]$, and every $b \in B_{\tau(z)}^{+}$ with $b \le z$ satisfies $b = P_z(b) \le P_z(z)$, so $P_z(z)$ is the largest element of $B_{\tau(z)}^{+} \cap [0,z]$. By (3),
\[
\sup\big(B^{+} \cap [0,z]\big) = \sup\big(B_{\tau(z)}^{+} \cap [0,z]\big) = P_z(z) \in B_{\tau(z)} \subseteq B.
\]
As $z$ was arbitrary, $B$ is a projection band, and the band projection of $z$ onto $B$ equals $P_z(z)$. Since $E$ is Archimedean, by the principal projection property, applying (1) to the principal band $B_{\tau(z)}$ (generated by $\tau(z)$) gives
\[
P_z(z) = \sup_{n} \big( z \wedge n\,\tau(z) \big),
\]
which completes the proof.
\end{proof}

\subsection{Proofs for Section~\ref{subsec:results-am-al}}

\begin{proof}[Proof of Lemma~\ref{lem:direct-sum-norms}]
Write $P_Y$ and $P_Z$ for the band projections onto $Y$ and $Z$, which are Riesz homomorphisms, and recall that lattice operations in $X$ are computed coordinatewise: for $y, y' \in Y$ and $z, z' \in Z$,
\[
(y+z) \wedge (y'+z') = (y \wedge y') + (z \wedge z'), \qquad (y+z) \vee (y'+z') = (y \vee y') + (z \vee z').
\]
In particular, $y+z$ and $y'+z'$ are disjoint if and only if $y \wedge y' = 0$ and $z \wedge z' = 0$.

Let $\|\cdot\|_Y$ and $\|\cdot\|_Z$ be lattice norms on $Y$ and $Z$, and set $\|y+z\| := \|y\|_Y + \|z\|_Z$ (resp.\ $\|y\|_Y \vee \|z\|_Z$). Then $\|\cdot\|$ is a lattice norm on $X$, since $|v| \le |w|$ implies $|P_Y v| = P_Y|v| \le P_Y|w| = |P_Y w|$, and likewise for $P_Z$. For disjoint positive elements $y+z$ and $y'+z'$ of $X$ one gets
\begin{align*}
\|(y+z)+(y'+z')\| &= \|y+y'\|_Y + \|z+z'\|_Z \\
&= \|y\|_Y+\|y'\|_Y+\|z\|_Z+\|z'\|_Z = \|y+z\|+\|y'+z'\|
\end{align*}
in the additive case, and
\begin{align*}
\|(y+z) \vee (y'+z')\| &= \|y \vee y'\|_Y \vee \|z \vee z'\|_Z \\
&= \|y\|_Y \vee \|y'\|_Y \vee \|z\|_Z \vee \|z'\|_Z = \|y+z\| \vee \|y'+z'\|
\end{align*}
in the $M$-norm case. So $\|\cdot\|$ is additive on disjoint positive elements, or an $M$-norm, whenever $\|\cdot\|_Y$ and $\|\cdot\|_Z$ are. It is also complete when $\|\cdot\|_Y$ and $\|\cdot\|_Z$ are: a $\|\cdot\|$-Cauchy sequence in $X$ has $\|\cdot\|_Y$-Cauchy and $\|\cdot\|_Z$-Cauchy components, by $\|P_Yv\|_Y, \|P_Zv\|_Z \le \|v\|$, and convergent components recombine. This gives the ``if'' parts of (i) and (ii).

Conversely, suppose $\|\cdot\|$ is an $AL$- or $AM$-norm on $X$. Its restrictions $\|\cdot\|_Y$ and $\|\cdot\|_Z$ to $Y$ and $Z$, which are norm closed since bands are norm closed in a Banach lattice, inherit the corresponding property. For $y\in Y$ and $z\in Z$, which are disjoint, $|y+z|=|y|+|z|=|y|\vee|z|$, so additivity, or the $M$-norm identity, of $\|\cdot\|$ forces $\|y+z\|=\|y\|_Y+\|z\|_Z$, or $\|y+z\|=\|y\|_Y\vee\|z\|_Z$. This proves (i) and (ii).

For (iii), let $\|\cdot\|$ be a lattice norm on $X$ with restrictions $\|\cdot\|_Y$ and $\|\cdot\|_Z$. Since $|P_Y v| \le |v|$ and $|P_Z v| \le |v|$,
\[
\max\big(\|P_Y v\|_Y, \|P_Z v\|_Z\big) \le \|v\| \le \|P_Y v\|_Y + \|P_Z v\|_Z \qquad (v \in X),
\]
so $\|\cdot\|$ is equivalent to the norm $y+z \mapsto \|y\|_Y + \|z\|_Z$ of the first part. Hence $X$ is complete if and only if $Y$ and $Z$ are, by the first part in one direction and by the norm closedness of $Y$ and $Z$ in the other. If $Y$ and $Z$ are $KB$-spaces and $(v_n)$ is an increasing norm bounded sequence in $X^{+}$, then $(P_Y v_n)$ and $(P_Z v_n)$ are increasing norm bounded sequences in $Y^{+}$ and $Z^{+}$, since $P_Y$ and $P_Z$ are positive, hence norm convergent, and $v_n = P_Y v_n + P_Z v_n$ converges. Conversely, if $X$ is a $KB$-space, an increasing norm bounded sequence in $Y^{+}$ converges in $X$, and its limit lies in the closed set $Y$, so $Y$ is a $KB$-space, and likewise $Z$.
\end{proof}

\begin{proof}[Proof of Theorem~\ref{thm:unitality-criterion}]
If $\tau$ is unital with truncation unit $u$, then $\tau(u) = u \wedge u = u$, so $u \in [\tau]$. Also $f = \tau(f) = u \wedge f \le u$ for every $f \in [\tau]$. Hence $\sup[\tau] = u$. Conversely, suppose $u := \sup[\tau]$ exists. The set $[\tau] = \tau(E^{+})$ is upward directed: given $\tau(f), \tau(g) \in [\tau]$, the element $\tau(f \vee g)$ lies in $[\tau]$ and dominates both, by Lemma~\ref{lem:truncation-basics}(i) and (ii). So the infinite distributive law for suprema gives, for every $f \in E^{+}$,
\[
u \wedge f = \sup_{g \in E^{+}} \big( \tau(g) \wedge f \big) = \sup_{g \in E^{+}} \big( g \wedge \tau(f) \big) = \tau(f),
\]
using $(\tau_1)$ in the form $\tau(g)\wedge f = g \wedge \tau(f)$ for the middle equality, and the fact that $g \wedge \tau(f) \le \tau(f)$ for all $g$, with equality at $g = \tau(f)$, for the last. (The distributive law holds in every Riesz space: if $v \ge \tau(g) \wedge f$ for all $g$, then $\tau(g) = \tau(g) \wedge f + \tau(g) \vee f - f \le v + u \vee f - f$ for all $g$, hence $u \le v + u \vee f - f$, that is, $u \wedge f = u + f - u \vee f \le v$.) So $\tau$ is unital with unit $u$.

For the consequence: if $E$ is a Banach lattice with the Levi property and $\tau$ is bounded, the upward directed set $[\tau] = \tau(E^{+})$ is norm bounded, so it has a supremum in $E$ by the Levi property, and $\tau$ is unital by the first part.

$KB$-spaces have the Levi property, by Theorem~\ref{thm:classical}(i). Dual Banach lattices have it as well: every norm bounded, upward directed net of positive functionals is also weak-* bounded, so it has a weak-* limit point that dominates it, and one checks this limit point is its supremum. Both $\ell^\infty$ and $L^\infty(\mu)$, for $\sigma$-finite $\mu$, are dual Banach lattices.
\end{proof}

\begin{proof}[Proof of Theorem~\ref{thm:c0-embeds}]
By Theorem~\ref{thm:unitality-criterion}, $E$ is not a $KB$-space (else the bounded truncation would be unital, contrary to hypothesis), so by Theorem~\ref{thm:classical}(ii) it contains a closed Riesz subspace lattice isomorphic to $c_0$. The remaining assertions follow since reflexivity and the $AL$ property each imply the $KB$ property (Theorem~\ref{thm:classical}(iii)), and weak sequential completeness is equivalent to it (Theorem~\ref{thm:classical}(ii)).
\end{proof}

\begin{proof}[Proof of Corollary~\ref{cor:never-AL}]
Suppose that $\|\cdot\|_U$ is a truncation unitization norm making $E \oplus \R$ an $AL$-space. If $E$ is unital with truncation unit $u$, then $u$ and $\mathbf 1 - u$ are disjoint positive elements of $E \oplus \R$ (Section~\ref{subsec:unitization}), so
\[
1 = \|\mathbf 1\|_U = \|u + (\mathbf 1 - u)\|_U = \|u\|_E + \|\mathbf 1 - u\|_U = 1 + \|\mathbf 1 - u\|_U,
\]
which forces $\mathbf 1 = u \in E$, a contradiction since $\mathbf 1 \notin E$. If $E$ is nonunital, then $E$ is a Banach lattice by Theorem~\ref{thm:extreme-norms}, and its norm, being the restriction of $\|\cdot\|_U$, is additive on disjoint positive elements, so $E$ is an $AL$-space with bounded truncation. This contradicts Theorem~\ref{thm:c0-embeds}.
\end{proof}

\begin{proof}[Proof of Theorem~\ref{thm:completion-unital}]
By Lemma~\ref{lem:birkhoff}(iii), $\tau$ is nonexpansive on $E^{+}$, hence uniformly continuous. Also $E^{+}$ is norm dense in $\overline E^{+}$. Indeed, if $f_n \to f$ in $\overline E$ with $f_n \in E$, then $f_n^{+} \to f^{+}$ by norm continuity of the lattice operations, and $f_n^+ \in E^+$.

So $\tau$ extends uniquely to a uniformly continuous $\overline\tau : \overline E^{+} \to \overline E^{+}$, the identity $(\tau_1)$ passes to the limit by continuity of the lattice operations, and $\sup\{\|\overline\tau(f)\| : f \in \overline E^{+}\} = 1$ follows from density of $E^+$ together with the normalization on $E$. Since every truncation on $\overline E$ is nonexpansive, again by Lemma~\ref{lem:birkhoff}(iii), $\overline\tau$ is the only truncation on $\overline E$ that extends $\tau$.

Now suppose $E$ is nonunital and $\|\cdot\|_E$ is additive on disjoint positive elements. For $x, y \in \overline E^{+}$ with $x \wedge y = 0$, choose $x_n, y_n \in E^{+}$ with $x_n \to x$ and $y_n \to y$. The elements $a_n := x_n - x_n \wedge y_n$ and $b_n := y_n - x_n \wedge y_n$ are positive and disjoint, and they converge to $x - x\wedge y = x$ and to $y$ respectively, so
\[
\|x+y\| = \lim_n \|a_n+b_n\| = \lim_n (\|a_n\|+\|b_n\|) = \|x\|+\|y\|.
\]
Hence $\overline E$ is an $AL$-space, so a $KB$-space by Theorem~\ref{thm:classical}(iii), and $\overline\tau$ is bounded by the first part. Theorem~\ref{thm:unitality-criterion} then makes $\overline\tau$ unital, with some truncation unit $\overline u$.

If $\overline u \in E$, then $\tau(f) = \overline\tau(f) = \overline u \wedge f$ for all $f \in E^+$ would make $\tau$ unital, contrary to hypothesis. Hence $\overline u \in \overline E \setminus E$.
\end{proof}

\begin{proof}[Proof of Theorem~\ref{thm:smallest-is-M-norm}]
Let $\varphi, \psi \in (E \oplus \R)^{+}$ with $\varphi \wedge \psi = 0$. Since $\|\cdot\|_S$ is a lattice norm and $\varphi, \psi \le \varphi \vee \psi$, we have $\|\varphi\|_S \vee \|\psi\|_S \le \|\varphi \vee \psi\|_S$.

For the reverse inequality, let $g \in E$ with $|g| \le \varphi \vee \psi$. By distributivity, $|g| = |g|\wedge(\varphi\vee\psi) = (|g|\wedge\varphi)\vee(|g|\wedge\psi)$. Put $g_1 = |g|\wedge\varphi$ and $g_2 = |g|\wedge\psi$. Since $E$ is an ideal of $E\oplus\R$ and $0 \le g_i \le |g| \in E^{+}$, both $g_1$ and $g_2$ lie in $E^{+}$, and they are disjoint, since $0 \le g_1 \wedge g_2 \le \varphi \wedge \psi = 0$. As $\|\cdot\|_E$ is an $M$-norm and $|g| = g_1 \vee g_2$,
\[
\|g\|_E = \|g_1 \vee g_2\|_E = \|g_1\|_E \vee \|g_2\|_E \le \|\varphi\|_S \vee \|\psi\|_S,
\]
the last inequality because $g_1 \le \varphi$, $g_2 \le \psi$, and $\|\cdot\|_S$ is defined as a supremum over such elements. Taking the supremum over admissible $g$ gives $\|\varphi \vee \psi\|_S \le \|\varphi\|_S \vee \|\psi\|_S$, completing the proof.
\end{proof}

\begin{proof}[Verification for Example~\ref{ex:S-vs-L}]
(i) The identification $c_0 \oplus \R \cong c$ follows from Section~\ref{subsec:unitization}, since $c_0$ is nonunital and $\mathbf 1$ is a weak order unit of $\ell^\infty$ implementing $\tau$. Every $g \in c_0$ with $|g|\le|f|$ satisfies $\|g\|_\infty \le \|f\|_\infty$, while truncating $f\in c$ to finitely many coordinates approximates $\|f\|_\infty$ from below within $c_0$. So $\|\cdot\|_S = \|\cdot\|_\infty$ on $c$. For $f \in c$ with limit $\ell$,
\[
\|f\|_L = \|(|f|-|\ell|\mathbf1)^{+}\|_\infty + |\ell| = (\sup_n|f_n| - |\ell|) + |\ell| = \|f\|_\infty
\]
since $\sup_n|f_n|\ge|\ell|$. So $\|\cdot\|_L = \|\cdot\|_\infty$ as well, and all truncation unitization norms on $c_0\oplus\R$ coincide with the supremum norm of $c$, which is a classical $M$-norm.

(ii) Identifying $E \oplus \R$ with $F = E + \R e \subset C_b(\R)$ via $x+r \mapsto x+re$, one computes, for $x \in E$, $r\ge0$ with $x+r\in(E\oplus\R)^{+}$,
\[
\|x+r\|_S = \|x+re\|_\infty, \qquad \|x+r\|_L = \|x^{+}\|_\infty + r
\]
(The first holds since $\|g\|_\infty \le \|x+re\|_\infty$ for $|g|\le x+re$, with near-equality realized by truncating $x+re$ to a large compact interval. The second follows from Theorem~\ref{thm:extreme-norms}, since $|x+r|-r\mathbf1 = x$ in $E\oplus\R$.) These formulas also give the equivalence of the two norms: $x + re \ge x^{+}$ and $(x+re)(t) \to r$ as $t \to -\infty$, so $\|x+r\|_L \le 2\|x+r\|_S$.

Taking $x(t) = (1-|t-2|)^{+}$ and $\varphi = x+1$, direct computation of $x+e$ on $(-\infty,0], [0,1], [1,2], [2,3], [3,\infty)$ gives
\[
\|\varphi\|_S = x(2)+e(2) = 4/3 \qquad \text{while} \qquad \|\varphi\|_L = \|x\|_\infty+1 = 2,
\]
so $\|\cdot\|_S \ne \|\cdot\|_L$.

Next, choose $k \in E$ with $0\le k\le1$ and $k=1$ on $[1,3]$, supported in $[1/2,7/2]$. Set $\psi = \mathbf1 - ek$, corresponding to $e(1-k)$. This vanishes on $[1,3] \supseteq \operatorname{supp} x$, so $x\wedge\psi = 0$. One finds that $x \vee \psi$ corresponds to $(x-ek)+\mathbf1$, with
\[
\|x\vee\psi\|_L = \|(x-ek)^{+}\|_\infty+1 = (x(2)-e(2))+1 = 5/3.
\]
Meanwhile $\|x\|_L = 1$ and $\|\psi\|_L = \|(-ek)^{+}\|_\infty+1=1$, so $\|x\vee\psi\|_L = 5/3 > 1 = \|x\|_L\vee\|\psi\|_L$. So $\|\cdot\|_L$ is not an $M$-norm. By contrast, $\|x\vee\psi\|_S = \|x+e(1-k)\|_\infty = 1 = \|x\|_S\vee\|\psi\|_S$, consistent with Theorem~\ref{thm:smallest-is-M-norm}.
\end{proof}

\subsection{Proofs for Section~\ref{subsec:results-hom}}

\begin{proof}[Proof of Lemma~\ref{lem:c0-hom}]
For $f,g \in c_0^+$, $(f\wedge\mathbf1)\wedge g = f\wedge g\wedge\mathbf1 = f\wedge(g\wedge\mathbf1)$ in $\R^\N$, so $\tau_E$ is a truncation.

If $f \in c_0^+$ satisfies $\tau_E(nf)=nf$ for all $n$, then $nf \le \mathbf1$, i.e., $f \le \tfrac1n\mathbf1$ for all $n$, so $f=0$: $\tau_E$ is Archimedean.

For $0 \ne f \in c_0^+$, $f$ is bounded, so $\lambda f \le \mathbf 1$ for $\lambda = 1/\sup_n f(n) > 0$, giving $\tau_E(\lambda f)=\lambda f$: $\tau_E$ is strong.

Finally, $\widetilde{\tau_E}(f) = \tau_E(f^+) - f^- = (f^+\wedge\mathbf1)-f^-$, so
\[
f - \widetilde{\tau_E}(f) = f^+ - (f^+\wedge\mathbf1) = (f^+-\mathbf1)^+,
\]
which vanishes off the finite set $\{n : f^+(n) \ge 1\}$ (as $f^+ \in c_0$), hence lies in $c_{00}$.
\end{proof}

\begin{proof}[Proof of Theorem~\ref{thm:hom-counterexample}]
$T$ is linear and surjective since $\pi$ is. For $f \in c_0$, Lemma~\ref{lem:c0-hom} gives $f - \widetilde{\tau_E}(f) \in c_{00} = \ker\pi$, so
\[
T(\widetilde{\tau_E}(f)) = -\pi(\widetilde{\tau_E}(f)) = -\pi(f) = T(f).
\]
Also $\widetilde{\tau_F} = \mathrm{id}_F$, by Lemma~\ref{lem:identity-truncation} applied to $F$, so $\widetilde{\tau_F}(T(f)) = T(f)$. Hence $T \circ \widetilde{\tau_E} = \widetilde{\tau_F}\circ T$, and $T$ is a truncation homomorphism.

For the sequence $f = (1,\tfrac12,\tfrac13,\dots) \in c_0^+$, $f \notin c_{00}$, so $[f] := \pi(f) > 0$ in $c_0/c_{00}$ (as $\pi$ is a Riesz homomorphism with kernel $c_{00}$) and $T(f) = -[f] < 0$. Then
\[
T(f\vee 0) = T(f) = -[f] \ne 0 = (-[f])\vee 0 = T(f)\vee T(0),
\]
so $T$ fails to preserve finite suprema (hence is not a Riesz homomorphism) and, the same computation showing $T(f) < 0 \le T(0)$ with $f \ge 0$, $T$ is not positive.
\end{proof}

\begin{proof}[Proof of Lemma~\ref{lem:test-domain}]
Since $f\vee h = (f-h)^++h$, closure of $D$ under $f \mapsto f^+$ suffices to show that $D$ is a Riesz subspace. Write $f = \alpha u + g \in D$. If $\alpha>0$, then $f$ agrees with $\alpha u>0$ off the finite set $\operatorname{supp}g$, so $f^- \in c_{00}$ and $f^+ = \alpha u + (g+f^-) \in D$. If $\alpha<0$, then $f^+ \in c_{00} \subseteq D$. If $\alpha=0$, then $f^+ = g^+ \in c_{00}$.

For the truncation claims, Lemma~\ref{lem:c0-hom} applied inside $c_0$ shows $f\wedge\mathbf1 = f-(f-\mathbf1)^+$ with $(f-\mathbf1)^+ \in c_{00} \subseteq D$ for $f \in D^+$, so $\tau_D$ maps $D^+$ into $D^+$, and the truncation identity, Archimedeanness, strongness, and the formula for $f-\widetilde{\tau_D}(f)$ all follow exactly as in Lemma~\ref{lem:c0-hom}, computed inside $D$.
\end{proof}

\begin{proof}[Proof of Theorem~\ref{thm:rank-one}]
$T$ is linear, being $\lambda$ composed with $t \mapsto -tg_0$. By Lemma~\ref{lem:test-domain}, $f - \widetilde{\tau_D}(f) \in c_{00} \subseteq \ker\lambda$ for $f \in D$, so $\lambda(\widetilde{\tau_D}(f)) = \lambda(f)$ and $T(\widetilde{\tau_D}(f)) = -\lambda(f)g_0 = T(f)$. Since $T(f)$ is a scalar multiple $c g_0$ of $g_0 \in F_0(\tau_F)$, we have $\widetilde{\tau_F}(T(f)) = T(f)$. Indeed, for $c \ge 0$ we get $c g_0 \le \lceil c \rceil g_0 \in [\tau_F]$, so $\tau_F(c g_0) = c g_0$ by Lemma~\ref{lem:truncation-basics}(iii), while for $c < 0$ we get $\widetilde{\tau_F}(c g_0) = \tau_F(0) - |c| g_0 = c g_0$. Hence $T\circ\widetilde{\tau_D} = \widetilde{\tau_F}\circ T$.

Finally $u \in D^+$, $\lambda(u)=1$, so $T(u\vee0)=T(u)=-g_0 < 0 = T(u)\vee T(0)$: $T$ is neither a Riesz homomorphism nor positive.
\end{proof}

\begin{proof}[Proof of Theorem~\ref{thm:optimality}]
(i)$\Rightarrow$(ii) is Theorem~\ref{thm:BH}. Implication (ii)$\Rightarrow$(iii) holds since Riesz homomorphisms are positive. Implications (ii)$\Rightarrow$(iv) and (iii)$\Rightarrow$(v) follow by specializing $E = D$.

Finally, suppose (i) fails. Theorem~\ref{thm:rank-one} then gives a truncation homomorphism $(D,\tau_D)\to F$ that is neither a Riesz homomorphism nor positive, so both (iv) and (v) fail. This gives (iv)$\Rightarrow$(i) and (v)$\Rightarrow$(i).
\end{proof}

\begin{proof}[Proof of Proposition~\ref{prop:induced-truncation}]
If $g_1,g_2 \in F^+$ with $q(g_1)=q(g_2)$, then $|g_1-g_2| \in B^+$, so $\tau(|g_1-g_2|) = |g_1-g_2|$, this being the $n=1$ case of membership in $B$. Lemma~\ref{lem:birkhoff}(iii) then gives $|\tau(g_1)-\tau(g_2)| \le \tau(|g_1-g_2|) = |g_1-g_2| \in B$. Since $B$ is solid, this gives $\tau(g_1)-\tau(g_2) \in B$, that is, $q(\tau(g_1))=q(\tau(g_2))$. Since $(F/B)^+ = q(F^+)$, this defines $\overline\tau$ unambiguously.

For $x=q(f), y=q(g) \in (F/B)^+$ ($f,g\in F^+$), $q$ being a Riesz homomorphism gives $\overline\tau(x)\wedge y = q(\tau(f)\wedge g) = q(f\wedge\tau(g)) = x \wedge \overline\tau(y)$, so $\overline\tau$ is a truncation.

Finally, for $h \in F$,
\[
q(\widetilde\tau(h)) = q(\tau(h^+))-q(h^-) = \overline\tau(q(h)^+)-q(h)^- = \widetilde{\overline\tau}(q(h))
\]
(using $q(h^{\pm})=q(h)^{\pm}$), so $q$ is a truncation homomorphism, and it is a Riesz homomorphism by construction.
\end{proof}

\begin{proof}[Proof of Lemma~\ref{lem:reduction}]
The composition of two truncation homomorphisms is a truncation homomorphism (immediate from the definition), so $q\circ T$ is a truncation homomorphism from $(E,\tau_E)$ into $(F/B,\overline\tau)$ by Proposition~\ref{prop:induced-truncation}. As $\overline\tau$ is Archimedean by hypothesis, Theorem~\ref{thm:BH} gives that $q\circ T$ is a Riesz homomorphism.
\end{proof}

\begin{proof}[Proof of Theorem~\ref{thm:main-hom}]
Write $\tau(f) = pf + e\wedge f$ for $f \in F^+$, via the structure theorem, with $p$ a component of the weak unit $w$ of $F^u$ and $e \perp p$ in $(F^u)^+$. For $h \in (F^u)^+$, write $h = h_p + h_d$ with $h_p = ph$ and $h_d = (w-p)h$. Multiplication by $p$ is the band projection onto the band $B_p$ generated by $p$, and lattice operations decompose along $F^u = B_p \oplus \{p\}^d$. Indeed, since $p \wedge (w-p) = 0$, the $f$-algebra property gives $ph \perp w-p$ and $(w-p)h \perp p$ for $h \in (F^u)^{+}$, while $\{p\}^d = B_{w-p}$ and $\{w-p\}^d = B_p$ because $w$ is a weak order unit of the Dedekind complete space $F^u$, so that $h = ph + (w-p)h$ is the decomposition of $h$ along $B_p \oplus \{p\}^d$. Now let $f \in F^+$ and $n \ge 1$. As $e \perp p$, also $e \perp nf_p$, so identity (2) gives $e \wedge nf = e \wedge (nf_p + nf_d) = e \wedge nf_d$, and therefore
\[
nf - \tau(nf) = nf_d - e \wedge nf_d = (nf_d - e)^+ \in \{p\}^d.
\]
From this one shows $F_0(\tau) = F \cap B_p$, both sets being solid, so that it suffices to compare their positive parts. If $f \in F_0(\tau)^+$, then $nf_d \le e$ for all $n$, and $F^u$ Archimedean forces $f_d=0$, that is, $f = pf \in B_p$. Conversely, $f \in F^+\cap B_p$ gives $f_d=0$, hence $nf-\tau(nf)=0$ for all $n$.

Now let $x = q(g) \in F_0(\overline\tau)^+$, with $g\in F^+$. Since $\overline\tau(nq(g)) = nq(g)$, we get $ng-\tau(ng) \in B := F_0(\tau)$ for all $n$, so
\[
ng - \tau(ng) \in B\cap\{p\}^d = (F\cap B_p)\cap\{p\}^d = \{0\}
\]
(as $B_p \cap \{p\}^d = \{0\}$), that is, $ng_d \le e$ for all $n$. Archimedeanness of $F^u$ then gives $g_d=0$, so $g \in F\cap B_p = B$, that is, $x = q(g) = 0$. Hence $F_0(\overline\tau) = \{0\}$, and $\overline\tau$ is Archimedean.
\end{proof}

\begin{proof}[Proof of Corollary~\ref{cor:modulo-infinitesimals}]
Immediate from Theorem~\ref{thm:main-hom} and Lemma~\ref{lem:reduction}. The displayed reformulation follows since $q(T(f\vee g)) = q(Tf)\vee q(Tg) = q(Tf\vee Tg)$ if and only if $T(f\vee g)-Tf\vee Tg \in \ker q = F_0(\tau_F)$. Likewise, $q(Tf)\ge0$ if and only if $q((Tf)^-) = 0$, if and only if $(Tf)^- \in F_0(\tau_F)$.
\end{proof}

\begin{proof}[Proof of Proposition~\ref{prop:totally-ordered}]
Here $\tau(f)=v\wedge f$, so $B=F_0(\tau) = \{x \in F : n|x|\le v \text{ for all } n\}$. Let $x=q(g)\in F_0(\overline\tau)^+$, with $g\in F^+$. For every $n$, $\overline\tau(nx) = nx$ reads $q(ng - \tau(ng)) = 0$, that is, $(ng-v)^+ = ng - v \wedge ng \in B$. Suppose that $mg \not\le v$ for some $m$. Total ordering gives $v < mg$, so $h:=(mg-v)^+=mg-v>0$ lies in $B$. Taking $n=2m$ and using $2mg-v=v+2h\ge0$ gives $v+2h \in B$. Solidity of $B$, with $0\le v\le v+2h$, then forces $v \in B$, that is, $nv\le v$ for all $n$, so $v=0$. But then $\tau$ is the zero truncation with $B=\{0\}$, contradicting $h>0$ and $h\in B$.

Hence $ng\le v$ for all $n$, i.e.\ $g\in B$ and $x=0$.
\end{proof}

\begin{proof}[Proof of Proposition~\ref{prop:ideal-condition}]
Since $T(E)\subseteq I$ and $I$ is an ideal, it is closed under $\tau_F$, because $\tau_F(f) \le f$ for $f \in I^{+}$ (Lemma~\ref{lem:truncation-basics}(i)), so $\tau_F|_{I^+}$ is a truncation on $I$. By hypothesis, $F_0(\tau_F|_{I^+}) = I\cap F_0(\tau_F) = \{0\}$, that is, $\tau_F|_{I^+}$ is Archimedean. Viewed as a map into $I$, $T$ is still a truncation homomorphism, since the extended truncation of $\tau_F|_{I^+}$ is the restriction of $\widetilde{\tau_F}$ to $I$.

Theorem~\ref{thm:BH} then makes $T : E \to I$ a Riesz homomorphism, and since $I$ inherits its lattice operations from $F$, so is $T : E \to F$.
\end{proof}

\section*{Declarations}

\noindent\textbf{Funding.} The authors did not receive support from any organization for the submitted work.

\noindent\textbf{Competing interests.} The authors have no relevant financial or non-financial interests to disclose.

\noindent\textbf{Data availability.} Data sharing is not applicable to this article, as no datasets were generated or analysed during the current study.

% BIBLIOGRAPHY
% ==============================================================

\end{document}